\documentclass[a4paper,10pt]{amsart}

\pdfoutput=1

\usepackage{amsmath,amssymb,amsthm,url}
\usepackage[utf8]{inputenc}
\usepackage[T1]{fontenc}
\usepackage{libertine}
\usepackage[libertine,cmintegrals,cmbraces]{newtxmath}
\usepackage{bbold}
\usepackage[shortlabels]{enumitem}
\usepackage{mathtools}
\usepackage{adjustbox}
\usepackage{xcolor}
\usepackage{tikz}
\usepackage{tikz-cd}
\usepackage{nicefrac}
\usepackage{dsfont}
\usepackage{todonotes}
\usepackage{a4wide}
\usepackage{quiver}
\usepackage{anyfontsize}
\usepackage{euscript}
\usepackage[all]{xy}
\usepackage{thmtools}
\usepackage{hyperref}
\hypersetup{%
  bookmarksnumbered=true,
  colorlinks=true,%
  linkcolor=blue,%
  citecolor=blue,%
  filecolor=blue,%
  menucolor=blue,%
  urlcolor=blue,%
  pdfnewwindow=true,%
  pdfstartview=FitBH}
\usepackage[capitalise]{cleveref}

\usepackage[
  style=alphabetic,
  backend=biber,
  sorting=nyt,
  isbn=false,
  maxbibnames=99
]{biblatex}

\AtEveryBibitem{\clearlist{language}}
\renewbibmacro{in:}{}

\usepackage{xcolor}
\definecolor{seagreen}{RGB}{46,139,87}
\definecolor{maroon}{RGB}{128,0,0}
\definecolor{darkviolet}{RGB}{148,0,211}
\definecolor{twelve}{RGB}{100,100,170}
\definecolor{thirteen}{RGB}{100,150,50}
\definecolor{fourteen}{RGB}{200,0,0}
\definecolor{fifteen}{RGB}{0,200,0}
\definecolor{sixteen}{RGB}{0,0,200}
\definecolor{seventeen}{RGB}{200,0,200}
\definecolor{eighteen}{RGB}{0,200,200}

\newcommand{\bb}[1]{\mathbb{#1}}

\newcommand{\es}[1]{\EuScript{#1}}
\renewcommand{\sf}[1]{{\mathsf{#1}}}

\DeclareMathOperator{\colim}{\mathsf{colim}}

\newcommand{\s}{{\sf{Sp}}}

\newcommand{\Fin}{\sf{Fin}}
\newcommand{\Set}{\sf{Set}}

\newcommand{\Fun}{\sf{Fun}}

\DeclareMathOperator{\Map}{\mathsf{Map}}

\newcommand{\op}{\mathsf{op}}

\renewcommand{\inf}{\mathsf{inf}}

  \newcommand{\adjunction}[4]{
\xymatrix{
#1:#2 \ar@<.5ex>[r] &
\ar@<.5ex>[l] #3:#4
}}

\newtheorem{thm}{Theorem}[section]

\newtheorem{lem}[thm]{Lemma}
\newtheorem{cor}[thm]{Corollary}

\newtheorem{xxthm}{Theorem}

\theoremstyle{definition}
\newtheorem{definition}[thm]{Definition}
\newtheorem{ex}[thm]{Example}

\newtheorem{rem}[thm]{Remark}

\begin{document}
\title{Deconstructing span categories for profinite groups}

\author{David Barnes}
\address[Barnes]{Queen's University Belfast}
\email{d.barnes@qub.ac.uk}

\author{Niall Taggart}
\address[Taggart]{Queen's University Belfast}
\email{n.taggart@qub.ac.uk}

\begin{abstract}
One of the major advantages of $\infty$-category theory over classical $1$-category theory is its robust and homotopically meaningful framework for taking (co)limits of diagrams of $\infty$-categories. However, it is both subtle and crucial to specify which variant of the $\infty$-category of $\infty$-categories is being used when forming such (co)limits. In this article, we present a concrete case study illustrating how (co)limits of $\infty$-categories behave in a specific setting. We demonstrate that the span category of a profinite group can be realised as the colimit of the span categories of its quotients by open normal subgroups and provide a number of applications to the world of equivariant (higher) algebra.
\end{abstract}
\maketitle

\section{Introduction}

Since the inception of $\infty$-category theory one thing the community has always agreed upon was that a major advantage of $\infty$-categories over more classical approaches to homotopy theory is the ability to consider limits and colimits of $\infty$-categories in a homotopically meaningful and consistent fashion. For example, the literature is abound with tools and techniques for putting model structures on homotopy limits of diagrams of model categories under rather technical constraints, and the idea of homotopy colimits of model categories to this day remains subtle to the point of there being no good theory yet developed, but in the world of $\infty$-categories, these (co)limits becomes formal once you have a good $\infty$-category of $\infty$-categories. The latter exists, unlike a good model category of model categories.

Consider, for example, the stable homotopy category, $\es{SH}$. Since the dawn of time, homotopy theorists have agreed that the stable homotopy category should be formed from the category of pointed spaces $\es{S}_\ast$ by formally inverting the spheres. One possibly interpretation of this characterisation is that the stable homotopy category is 
\[
\es{SH} = \lim (\cdots \xrightarrow{\ \Omega \ } \es{S}_\ast \xrightarrow{\ \Omega \ } \es{S}_\ast)
\]
the limit of iterated applications of the loops endofunctor on pointed spaces. It is possible to prove this statement within the world of model categories, but in many ways such a proof largely reduces to luck: we are lucky that there is a convenient model category of pointed spaces satisfying the technical conditions requires for a homotopy limit model structure, we are lucky that iterating loops provides a sequential diagram of right Quillen functors, lucky that there exists a good model category of spectra, and lucky that we may construct, by hand, a Quillen equivalence between a model category of spectra and the homotopy limit model structure. 

In the $\infty$-category world we can still see the stable homotopy category as inverting spheres even on days when our luck has run out. We can define the $\infty$-category of spectra $\s$ to be 
\[
\s \coloneq \lim (\cdots \xrightarrow{\ \Omega \ } \es{S}_\ast \xrightarrow{\ \Omega \ } \es{S}_\ast)
\]
the limit in the $\infty$-category of $\infty$-categories of the loops functor on the $\infty$-category of pointed spaces, and show with little work that the underlying $\infty$-category of any good model category of spectra satisfies the same universal property as our limit construction: both are the free stable $\infty$-category generated under colimits by the one-point space.

The limit above is taken in the $\infty$-category of $\infty$-categories, but a little more is true. For instance, the $\infty$-category of pointed spaces is presentable, and the loops functor is a right adjoint, so this limit could equally well be computed in the $\infty$-category $\sf{Pr}^R$ of presentable $\infty$-categories and right adjoints. Passing to the associated diagram of left adjoints gives a diagram in the $\infty$-category $\sf{Pr}^L$ of presentable $\infty$-categories and left adjoints, and we can interpret the $\infty$-category of spectra as
\[
\s \simeq \colim (\es{S}_\ast \xrightarrow{\ \Sigma \ } \es{S}_\ast \xrightarrow{\ \Sigma \ } \cdots)
\]
the colimit in $\sf{Pr}^L$  of iterated applications of the suspension functor. This latter description is much more fitting with the idea that the stable homotopy category is given by inverting the spheres. 

The practice of taking (co)limits of $\infty$-categories in some variant of the $\infty$-category of $\infty$-categories, e.g., $\sf{Cat}_\infty$, $\sf{Pr}^L$, or $\sf{Pr}^R$, is now ubiquitous throughout homotopy theory with striking applications, including to: Hermitian $K$-theory~\cite{HermitianI, HermitianII, HermitianIII}, equivariant stable homotopy theory~\cite{LNP, BBB}, and continuous $K$-theory~\cite{Efimov}.  

Despite the ubiquity of these constructions, a down-to-earth approachable, example driven exposition on these constructions is absent from the literature. It is precisely this gap that this article aims to address. We provide a general introduction to the theory of limits and colimits of $\infty$-categories, highlighting how limits and colimits are computed in many of the variants of the $\infty$-category of $\infty$-categories. 

To show how the reader may apply this theory we have a single leading example, the span category of a profinite group. A profinite group is a compact, Hausdorff, totally disconnected topological group. These topological groups admit a more algebraic description: namely a profinite group $G$ may be deconstructed as
\[
G \cong \lim_{N \leq_o G} G/N
\]
a limit of its quotients by its open normal subgroups.  Our first application of the study of (co)limits of $\infty$-categories is a categorification of this result to categories of $G$-sets. 

\begin{xxthm}[\cref{thm: GSet as colim}]\label{main thm: GSet colim}
Let $G$ be a profinite group. The collection of inflation functors\footnote{The $N$-inflation functor arises from the quotient of groups $G \to G/N$ giving a finite $G/N$-set an action of $G$ by precomposition.} running over the set of open and normal subgroups of $G$ assembles into an equivalence of $\infty$-categories
\[
\underset{N \trianglelefteq_o G}{\colim}~\sf{Fin}_{G/N} \simeq \sf{Fin}_G^\delta,
\]
between the colimit (in $\sf{Cat}_\infty$) of the $\infty$-categories of finite $G/N$-sets and the $\infty$-category of finite \emph{discrete} $G$-sets.
\end{xxthm}

This may seem in contrast to the limit description of a profinite group, but upon universally adjoining filtered colimits, i.e., passing to Ind-completions, and examining how this process interacts with (co)limits of $\infty$-categories, we obtain a description analogous to that of a profinite group as a limit of its quotients by open normal subgroups.

\begin{xxthm}[{\cref{cor: GSet as limit}}]\label{main thm: GSet lim}
Let $G$ be a profinite group. The collection of fixed points functors\footnote{The $N$-fixed points functor sends a finite $G$-set to the set of elements fixed by every element of $N$, this latter set has a residual action by $G/N$.} assembles into an equivalence of $\infty$-categories 
\[
\Set_G^\delta \simeq \underset{N \trianglelefteq_o G}{\lim}~\Set_{G/N}
\]
between the $\infty$-category of discrete $G$-sets and the limit (in $\sf{Pr}^R$) of the $\infty$-categories of $G/N$-sets. Dually, there is an equivalence of $\infty$-categories
\[
\underset{N \trianglelefteq_o G}{\colim}~\Set_{G/N} \simeq \Set_G^\delta
\]
between the colimit (in $\sf{Pr}^L$) of the $\infty$-categories of  $G/N$-sets and the $\infty$-category of discrete $G$-sets.
\end{xxthm}

Another key player in many recent advances in homotopy theory is the $\infty$-category of spans, sometimes also referred to as the category of correspondences or the effective Burnside category. This is the $\infty$-categorical enhancement of Lindner's \cite{Lindner} $(2,1)$-category of spans. Applications abound much of modern homotopy theory including: algebraic $K$-theory \cite{BGT}, equivariant homotopy theory \cite{BarwickGlasmanShah}, motivic homotopy \cite{BachmannHoyois}, higher algebra \cite{Harpaz} and six-functor formalisms \cite{Scholze}. We also give an accessible survey of the construction of the $\infty$-category of spans and show that when considered as a functor between suitable variants of the $\infty$-category of $\infty$-categories it commutes with limits and filtered colimits. This analysis when combined with~\cref{main thm: GSet colim} and~\cref{main thm: GSet lim} proves our next main theorem.

\begin{xxthm}[\cref{thm: span as limit} and \cref{thm: span as colim}]
Let $G$ be a profinite group. The collection of fixed point functors running over the set of open and normal subgroups of $G$ assembles into an equivalence of $\infty$-categories
\[
\sf{Span}(\Set_G^\delta) \simeq \underset{N \trianglelefteq_o G}{\lim}~\sf{Span}(\Set_{G/N})
\]
between the limit (in $\sf{Cat}_\infty$) of the $\infty$-categories of $G/N$-sets and the $\infty$-category of discrete $G$-sets. Moreover, the collection of inflation functors running over the set of open and normal subgroups of $G$ assembles into an equivalence of $\infty$-categories
\[
\underset{N \trianglelefteq_o G}{\colim}~\sf{Span} (\Fin_{G/N}) \simeq \sf{Span}(\Fin_G^\delta),
\]
between the colimit (in $\sf{Cat}_\infty$) of the $\infty$-categories of finite $G/N$-sets and the $\infty$-category of finite discrete $G$-sets. These equivalences are not dual to each other.
\end{xxthm}

We give an application of our narrative to the study of equivariant (higher) algebra. In stable homotopy theory, there is a fully faithful embedding of abelian groups into the stable homotopy category. This follows from the fact that to any abelian group $A$ one may assign the Eilenberg-MacLane spectrum $HA$, whose homotopy groups recover $A$. In the equivariant world, abelian groups are replaced with Mackey functors of abelian groups. These are systems of abelian groups related by restriction, transfer and conjugation maps. For instance, a Mackey functor of abelian groups $M$ for the cyclic group $C_6$ of order $6$ may be presented via a Lewis diagram:
  
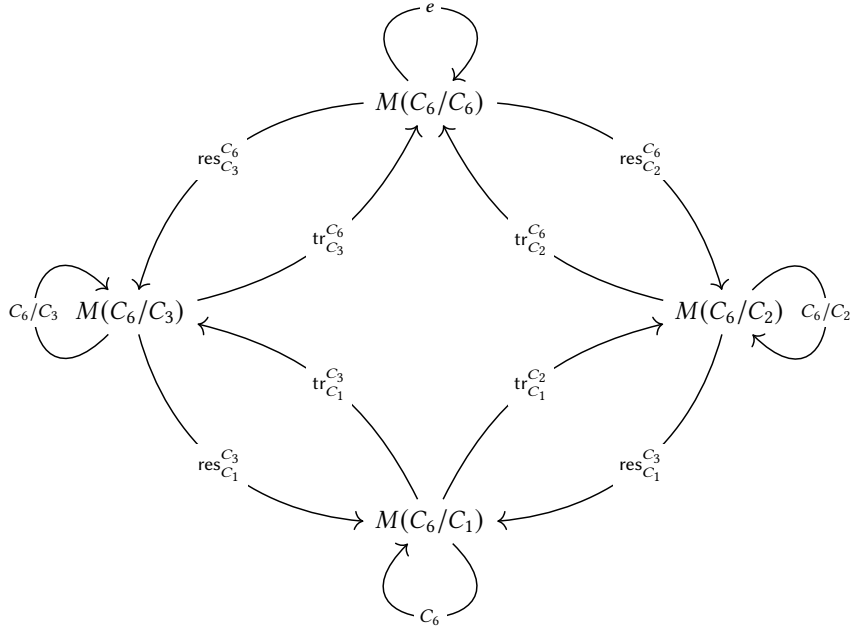
\[
\begin{tikzcd}[row sep=2.2cm, column sep=2.2cm]
& M(C_6/C_6)
  \arrow[loop above, "e" description, looseness=8, out=135, in=45]
  \arrow[dr, "\sf{res}_{C_2}^{C_6}" description, bend left=35]
  \arrow[dl, "\sf{res}_{C_3}^{C_6}" description, bend right=35]
\\
M(C_6/C_3)
  \arrow[loop left, "C_6/C_3" description, looseness=8, out=225, in=135]
  \arrow[ur, "\sf{tr}_{C_3}^{C_6}" description, bend right=25]
  \arrow[dr, "\sf{res}_{C_1}^{C_3}" description, bend right=35]
&&
M(C_6/C_2)
  \arrow[loop right, "C_6/C_2" description, looseness=8, out=45, in=315]
  \arrow[ul, "\sf{tr}_{C_2}^{C_6}" description, bend left=25]
  \arrow[dl, "\sf{res}_{C_1}^{C_3}" description, bend left=35]
\\
& M(C_6/C_1)
  \arrow[loop below, "C_6" description, looseness=8, out=315, in=225]
  \arrow[ur, "\sf{tr}_{C_1}^{C_2}" description, bend left=25]
  \arrow[ul, "\sf{tr}_{C_1}^{C_3}" description, bend right=25]
\end{tikzcd}
\]
The category $\sf{Mack}(G; \sf{Ab})$ of $G$-Mackey functors (of abelian groups) embeds fully faithfully into the genuine equivariant stable homotopy category $\es{SH}_G$ via the equivariant Eilenberg-MacLane spectrum functor. In a more higher algebraic formulation, the $\infty$-category of equivariant spectra is the $\infty$-category of $G$-Mackey functors of spectra, i.e., there is an equivalence
\[
\es{SH}_G \simeq \sf{Mack}(G; \es{SH})
\]
see e.g., work of Barwick~\cite{Barwick} or Guillou and May~\cite{GuillouMay}. In the case when the group of equivariance is a profinite group, the stable homotopy theory has been well studied by Fausk~\cite{Fausk} and the firth author and their collaborators~\cite{BarnesSugrueClassifying, BBB}. Our final main result gives an alternative approach and complementary perspective to these known approaches. Moreover, this example allows us to again highlight the nuances in working with different variants of the $\infty$-category of $\infty$-categories and reinforces our message on the subtly between (co)limits in $\sf{Cat}_\infty$, $\sf{Pr}^R$ and $\sf{Pr}^L$.

\begin{xxthm}[{\cref{thm: Mackey functors as a limit}}]
Let $\es{E}$ be a presentable semiadditive $\infty$-category and let $G$ be a profinite group. The collection of categorical fixed point functors running over the set of open and normal subgroups of $G$ assembles into an equivalence of $\infty$-categories
\[
\sf{Mack}(\Fin^\delta_G~; \es{E}) \simeq \lim_{N \leq_o G} \sf{Mack}(\Fin_{G/N}~; \es{E})
\]
between the limit (in $\sf{Pr}^R$) of the $\infty$-categories of $G/N$-Mackey functors and the $\infty$-category of $G$-Mackey functors. 
Dually, there is an equivalence of $\infty$-categories 
\[
\underset{N \leq_o G}{\colim}~\sf{Mack}(\Fin_{G/N}~; \es{E}) \simeq \sf{Mack}(\Fin^\delta_G~; \es{E})
\]
between the colimit (in $\sf{Pr}^L$) of the $\infty$-categories of $G/N$-Mackey functors and the $\infty$-category of $G$-Mackey functors. If $\es{E}$ is not presentable, then the limit equivalence holds in the $\infty$-category $\widehat{\sf{Cat}}_\infty$ of potentially large $\infty$-categories, but the colimit equivalence is false.
\end{xxthm}

\subsection*{Acknowledgements}
We are grateful to Kaya Arro and Bastiaan Cnossen for useful conversations. NT was supported by the Engineering and Physical Sciences Research Council under grant number EP/Z534705/1.

\subsection*{Conventions and Notation}
\begin{itemize}
    \item We denote by $\sf{Cat}_\infty$ the $\infty$-category of (essentially) small $\infty$-categories.
    \item We denote by $\sf{Pr}^L$ the $\infty$-category of presentable $\infty$-categories and left adjoints. Presentability here implies that being a left adjoint is equivalent to preserving all colimits. 
    \item We denote by $\sf{Pr}^R$ the $\infty$-category of presentable $\infty$-categories and right adjoints. Presentability here implies that being a right adjoint is equivalent to preserving all limits. 
    \item We denote by $\widehat{\sf{Cat}}_\infty$ the $\infty$-category of (possibly) large $\infty$-categories.
\end{itemize}

\section{(Co)Limits of \texorpdfstring{$\infty$-categories}{}}
This section is dedicated to understanding how limits and colimits of $\infty$-categories may be computed. We will begin with a discussion on how to compute these in the $\infty$-category of small $\infty$-categories, the discussion of which carries over directly to the $\infty$-category of large $\infty$-categories. We will then discuss how (co)limits may be computed in other variants of the $\infty$-category of $\infty$-categories as a relative statement with respect to  the $\infty$-category $\sf{Cat}_\infty$ of essentially small $\infty$-categories or the $\infty$-category $\widehat{\sf{Cat}}_\infty$ of potentially large $\infty$-categories.

\subsection{Limits of \texorpdfstring{$\infty$-categories}{}}\label{subsection: limits}
By~\cite[\href{https://kerodon.net/tag/02TL}{Corollary 02TL}]{kerodon} the $\infty$-category $\sf{Cat}_\infty$ of essentially small $\infty$-categories admits all small limits. We will describe how these limits are constructed, but for a more comprehensive account we direct the reader to~\cite[\href{https://kerodon.net/tag/02T0}{Subsection 02T0}]{kerodon} or \cite[\S3.3.3]{HTT}.

Let $I$ denote a small $\infty$-category and let $\es{C}_{(-)}: I \to \sf{Cat}_\infty$. By straightening-unstraightening\footnote{Also sometimes called the ``(Lurie-)Grothendieck construction'', or ``covariant transport''.} the data of the functor $\es{C}_{(-)}$ is equivalent to a data of a cocartesian fibration
\[
U: \textstyle\int_{I} \es{C}_{(-)} \longrightarrow I,
\]
where $\int_I \es{C}_{(-)}$ is the $\infty$-category of elements of $\es{C}_{(-)}$. The objects of $\int_{I} \es{C}_{(-)}$ are given by pairs $(i,X)$ where $i \in I$ and $X \in \es{C}_i$, see~\cite[\href{https://kerodon.net/tag/026V}{Example 026V}]{kerodon}, while the morphisms in $\int_{I} \es{C}_{(-)}$ from $(i,X)$ to $(j,Y)$ are given by pairs $(f,u)$ where $f : i \to j$ is a morphism in $I$ and $u: \es{C}_{(f)}(X) \to Y$ is a morphism in $\es{C}_{j}$, see~\cite[\href{https://kerodon.net/tag/026W}{Example 026W}]{kerodon}. A comprehensive account of the $\infty$-category of elements may be found in~\cite[\href{https://kerodon.net/tag/026H}{Subsection 026H}]{kerodon}. By~\cite[\href{https://kerodon.net/tag/02TK}{Corollary 02TK}]{kerodon} the limit of the functor $\es{C}_{(-)}: I \to \sf{Cat}_\infty$ is given by
\[
\lim_{i \in I} \es{C}_i \simeq \Fun_{/I}^{\rm{cocart}}(I, \textstyle \int_{I} \es{C}_{(-)} ) =: \Gamma(I, \int_{I} \es{C}_{(-)} ),
\]
the $\infty$-category of cocartesian sections of the cocartesian fibration $U :\int_{I} \es{C}_{(-)}  \to I$, i.e., sections of the fibration $U$ which send each morphism in $I$ to a $U$-cocartesian\footnote{A morphism $(f,u) : (i,X) \to (j,Y)$ is \emph{$U$-cocartesian} if and only if $u: \es{C}_{f}(X) \to Y$ is an equivalence in $\es{C}_{j}$, see~\cite[\href{https://kerodon.net/tag/026X}{Remark 026X}]{kerodon}.} morphism in $\int_{I} \es{C}_{(-)}$.

\subsection{Colimits of small \texorpdfstring{$\infty$-categories}{}}\label{subsection: colimits}
We now turn out attention to the case of colimit diagrams of $\infty$-categories. Our case of interest is exclusively in filtered colimits, but we will, for the sake of exposition, briefly discuss the general situation. By~\cite[\href{https://kerodon.net/tag/02V1}{Corollary 02V1}]{kerodon} the $\infty$-category $\sf{Cat}_\infty$ of (essentially) small $\infty$-categories admits all small colimits, with a more comprehensive account given in~\cite[\href{https://kerodon.net/tag/02UN}{Subsection 02UN}]{kerodon} or~\cite[\S3.3.4]{HTT}.

Given a functor $\es{C}_{(-)}: I \to \sf{Cat}_\infty$, we again represent $\es{C}_{(-)}$ by the associated cocartesian fibration $U : \int_I \es{C}_{(-)} \to I$. By~\cite[\href{https://kerodon.net/tag/02V0}{Corollary 02V0}]{kerodon} the colimit of the functor $\es{C}_{(-)}$ is given by
\[
\underset{i \in I}{\colim}~\es{C}_i \simeq \textstyle\int_I \es{C}_{(-)} [W^{-1}],
\]
the localization of the $\infty$-category $\int_I \es{C}_{(-)}$ of elements of $\es{C}_{(-)}$ at the set $W$ of $U$-cocartesian morphisms.

We now restrict to the case where $I$ is a filtered $\infty$-category. In fact, for our purposes it will suffice to consider the case when $I$ is a filtered $1$-category. For a comprehensive account of the theory of filtered colimits of $\infty$-categories we direct the reader to the account by Lurie~\cite[\href{https://kerodon.net/tag/062Y}{Subsection 062Y}]{kerodon} and the account by Rozenblyum~\cite{Rozenblyum}.

Let $\es{C}_{(-)}: I \to \sf{Cat}_\infty$ be a functor with $I$ a filtered $\infty$-category in the sense of~\cite[\href{https://kerodon.net/tag/02P9}{Definition 02P9}]{kerodon}. Objects of $\colim_{I} \es{C}_{(-)}$ are given as images of the objects of $\es{C}_i$ under the canonical functors $\es{C}_i \to \colim_{I} \es{C}_{(-)}$, i.e., 
\[
\sf{Ob}\left(\underset{i \in I}{\colim}~\es{C}_i\right) = \coprod_{i \in I} \sf{Ob}(\es{C}_i) / \sim
\]
where $X_i \sim X_j$ if and only if there exists a zigzag $i \xrightarrow{f} k \xleftarrow{g} j$ of maps in $I$ such that there is an equivalence $\es{C}_{f}(X_i) \simeq \es{C}_{g}(X_j)$ in $\es{C}_k$. Seeing this characterisation of the objects of $\colim_{i \in I}\es{C}_i$ is essentially an exercise on unravelling the definition of a colimit as a localization of the $\infty$-category of elements. In particular, an object $X$ of $\underset{i \in I}{\colim}~\es{C}_i$ occurs in some $\es{C}_i$. We denote the corresponding object by $X_i$ so that $X_i$ is mapped to $X$ under the canonical map $\es{C}_i \to \underset{i \in I}{\colim}~\es{C}_i$. Given two objects $X$ and $Y$ of the colimit $\underset{i \in I}{\colim}~\es{C}_i$, we can find an $i, j \in I$ such that $X$ is represented by $X_{i} \in \es{C}_{i}$ and $Y$ is represented by $Y_{j}$ in $\es{C}_{j}$. 
Since $I$ is filtered, for $k \in I$ with $k \geq i, j$, the object $X$ may be equivalently represented by $\es{C}_{f}(X_{i}) = X_k \in \es{C}_k$ for $f \colon i \to k$ and, the object $Y$ may be equivalently represented by $\es{C}_{g}(Y_{j}) = Y_k \in \es{C}_k$, for  $g \colon j \to k$ . It follows that the mapping space in $\underset{i \in I}{\colim}~\es{C}_i$ is given by
\[
\Map_{\underset{i \in I}{\colim}~\es{C}_i}(X,Y) = \colim_{k \in I^{(i,j)/}} \Map_{\es{C}_k}(X_k, Y_k).
\]
where $I^{(i, j)/}$ is the category of diagrams $i \rightarrow k \leftarrow j$ in $I$. Since the set of such $k$ is cofinal in $I$, one often writes
\[
\Map_{\underset{i \in I}{\colim}~\es{C}_i}(X,Y) = \underset{i \in I}{\colim}~ \Map_{\es{C}_i}(X_i, Y_i),
\]
see e.g.,see \cite[Lemma 0.2.1 and \S0.4]{Rozenblyum}.

\subsection{(Co)limits of potentially large \texorpdfstring{$\infty$-categories}{}}
Denote by $\widehat{\sf{Cat}_\infty}$ the $\infty$-category of potentially large $\infty$-categories. By passing to a larger universe (which we have kept implicit throughout) the above discussion extends to provide methods for computing (co)limits in $\widehat{\sf{Cat}_\infty}$, i.e., colimits are a localization of the Grothendieck construction and limits are given by cartesian sections of the fibration associated to the Grothendieck construction.

\subsection{(Co)limits of presentable \texorpdfstring{$\infty$-categories}{}}
Denote by $\sf{Pr}^L$ the $\infty$-category of presentable $\infty$-categories and left adjoints, and denote by $\sf{Pr}^R$ the $\infty$-category of presentable $\infty$-categories and right adjoints, see e.g.,~\cite[Definition 5.5.3.1]{HTT}. By \cite[Proposition 5.5.3.13 and Theorem 5.5.3.18]{HTT} $\sf{Pr}^L$ and $\sf{Pr}^R$ have all small limits and the inclusion functors\footnote{The \emph{inclusion} functors ``forget'' that the $\infty$-category is presentable and that the functor is a left (resp. right) adjoint.}
\begin{align*}
    \sf{Pr}^L &\hookrightarrow \widehat{\sf{Cat}_\infty} \\
    \sf{Pr}^R &\hookrightarrow \widehat{\sf{Cat}_\infty}
\end{align*}
preserves limits. The $\infty$-category $\sf{Pr}^L$ of presentable $\infty$-categories and left adjoints is dual to the $\infty$-category $\sf{Pr}^R$ of presentable $\infty$-categories and right adjoints, i.e., there is an equivalence  $(\sf{Pr}^L)^\op \simeq \sf{Pr}^R$ obtained by passing to adjoints, see e.g.,~\cite[Corollary 5.5.3.4]{HTT}. It follows that a cocone in $\sf{Pr}^L$ is a colimit cocone if and only if the cone given by taking right adjoints is a limit cone in $\sf{Pr}^R$. Said differently, a colimit in $\sf{Pr}^L$ is computed by passing to adjoints and computing the limit in $\sf{Pr}^R$, which in turn may be computed in $\widehat{\sf{Cat}_\infty}$.

In particular for $\es{C}_{(-)}: I \to \sf{Pr}^L$, given a collection of compatible maps $\{f_i: \es{C}_i \to \es{D}\}_{i \in I}$ in $\sf{Pr}^L$, there is an induced map
\[
\underset{i \in I}{\colim} \es{C}_i \longrightarrow  \es{D}.
\]
Passing to the right adjoints of $\{f_i\}_{i \in I}$ gives an induced map 
\[
\es{D} \longrightarrow  \underset{i \in I}{\lim} \es{C}_i
\]
in $\sf{Pr}^R$. From this discussion, the following lemma is immediate.

\begin{lem}\label{lem: cocone in PrL is cone in PrR}
Let $\es{C}_{(-)}: I \to \sf{Pr}^L$ be a diagram of presentable $\infty$-categories and let $\es{D}$ be a presentable $\infty$-category. If  $\{f_i: \es{C}_i \to \es{D}\}_{i \in I}$ is a compatible collection of maps in $\sf{Pr}^L$, then the induced map 
\[
\underset{i \in I}{\colim} \es{C}_i \longrightarrow \es{D}
\]
is an equivalence in $\sf{Pr}^L$ if and only if the map 
\[
\es{D} \longrightarrow  \underset{i \in I}{\lim} \es{C}_i,
\]
induced by the right adjoints of $\{f_i\}_{i \in I}$ is an equivalence in $\sf{Pr}^R$.
\end{lem}

\subsection{(Co)limits and functor categories}

We begin with a brief detour through the landscape of cartesian closed $\infty$-categories. A comprehensive account of such may be found in \cite[\href{https://kerodon.net/tag/05UE}{Subsection 05UE}]{kerodon}, but as with the spirit of this document, we only summarise the salient features for us. Essentially an $\infty$-category $\es{C}$ is \emph{cartesian closed} if it admits finite products and for every object $C$ of $\es{C}$, the functor 
\[
C \times (-) : \es{C} \longrightarrow \es{C}
\]
is a left adjoint. We will denote the right adjoint by $\es{C}(X, -): \es{C} \to \es{C}$, and refer to it as the internal hom object, see e.g.,\cite[\href{https://kerodon.net/tag/05UJ}{Definition 05UJ}]{kerodon} for a precise definition. Many of the $\infty$-categories that we have seen so far in this document are cartesian closed, for instance the $\infty$-category $\sf{Cat}_\infty$ of small $\infty$-categories is cartesian closed~\cite[\href{https://kerodon.net/tag/05UR}{Example 05UR}]{kerodon} as is the $\infty$-category $\widehat{\sf{Cat}_\infty}$ by an analogous argument, just working the ``large'' universe, but we strongly advise the reader to not think too long about these set theoretical issues. In both cases, the internal hom object is given by the $\infty$-category of functors $\Fun(C, -)$.

The following fact is surely well-known to the experts but we were unable to locate a proof in the literature, although the proof proceeds precisely as one would proceed in $1$-category world.

\begin{lem}
Let $\es{C}$ be a bicomplete cartesian closed $\infty$-category. Given a diagram $\es{C}_{(-)}: I \to \es{C}$, there is an equivalence
\[
\es{C}(\underset{i \in I}{\colim}~\es{C}_i, D) \simeq \lim_{i \in I} \es{C}(\es{C}_i, D)
\]
for each $D \in \es{C}$.
\end{lem}
\begin{proof}
We will employ a Yoneda style argument and show that the objects $\es{C}(\underset{i \in I}{\colim}~\es{C}_i, D)$ and $\underset{i \in I}{\lim} \es{C}(\es{C}_i, D)$ represent the same presheaf on $\es{C}$, i.e., that there is an equivalence
\[
\Map(C, \es{C}(\underset{i \in I}{\colim}~\es{C}_i, D)) \simeq \Map(C, \lim_{i \in I}  \es{C}(\es{C}_i, D))
\]
for every $C \in \es{C}$. Indeed, there is a series of equivalences
\begin{align*}
\Map(C, \es{C}(\underset{i \in I}{\colim}~\es{C}_i, D)) &\simeq \Map(C \times \underset{i \in I}{\colim}~\es{C}_i, D) \\ &\simeq  \Map(\underset{i \in I}{\colim}~(C \times \es{C}_i), D) \\ &\simeq \lim_{i \in I} \Map(C \times \es{C}_i, D) \\ &\simeq \lim_{i \in I} \Map(C, \es{C}(\es{C}_i, D)) \\ &\simeq \Map(C, \lim_{i \in I}  \es{C}(\es{C}_i, D))
\end{align*}
where the first, second and forth equivalences are essentially applications of the fact that $\es{C}$ is cartesian closed, while the third and fifth equivalences are the standard relationships between mapping spaces and (co)limits.
\end{proof}

\begin{lem}\label{lem: functor cats and colims}
Let $\es{C}_{(-)}: I \to \sf{Cat}_\infty$ be a digram of small $\infty$-categories. For any $\infty$-category $\es{D}$, there is an equivalence
\[
\Fun(\underset{i \in I}{\colim}~\es{C}_i, \es{D}) \simeq \lim_{i \in I} \Fun(\es{C}_i, \es{D})
\]
between the $\infty$-category of $\es{D}$-valued functors on $\underset{i \in I}{\colim}~\es{C}_i$ (computed in $\sf{Cat}_\infty$) and the limit (in $\widehat{\sf{Cat}_\infty}$) of the $\infty$-categories of $\es{D}$-valued functors on $\es{C}_i$.
\end{lem}
\begin{proof}
The inclusion functor $\sf{Cat}_\infty \hookrightarrow \widehat{\sf{Cat}_\infty}$ preserves colimits hence the proof reduces to showing that $\Fun(-,-)$ is an internal hom object for $\widehat{\sf{Cat}_\infty}$. This latter holds by (a size argument adjustment of)~\cite[\href{https://kerodon.net/tag/05UR}{Example 05UR}]{kerodon}.
\end{proof}

In the situation of the above lemma, if $\es{D}$ is presentable, then the $\infty$-category $\Fun(\es{C}_i, \es{D})$ is presentable by~\cite[Proposition 5.5.3.6]{HTT}. Moreover, the transition maps $f_{i}^j : \es{C}_i \to \es{C}_j$ in the colimit induce an adjoint pair 
\[
\adjunction{(f_i^j)_!}{\Fun(\es{C}_i, \es{D})}{\Fun(\es{C}_j, \es{D})}{(f_i^j)^\ast},
\]
and the limit above is taken over the restriction maps. It follows that the diagram of functor categories is a diagram in $\sf{Pr}^R$, and so the limit may be computed there. 

\begin{cor}
Let $\es{C}_{(-)}: I \to \sf{Cat}_\infty$ be a digram of small $\infty$-categories. For any presentable $\infty$-category $\es{D}$, there is an equivalence
\[
\Fun(\underset{i \in I}{\colim}~\es{C}_i, \es{D}) \simeq \lim_{i \in I} \Fun(\es{C}_i, \es{D})
\]
between the $\infty$-category of $\es{D}$-valued functors on $\underset{i \in I}{\colim}~\es{C}_i$ (computed in $\sf{Cat}_\infty$) and the limit (in $\sf{Pr}^R$) of the $\infty$-categories of $\es{D}$-valued functors on $\es{C}_i$.
\end{cor}

Passing to adjoints, i.e., using that limits in $\sf{Pr}^R$ are colimits in $\sf{Pr}^L$, yields the following.

\begin{cor}
Let $\es{C}_{(-)}: I \to \sf{Cat}_\infty$ be a digram of small $\infty$-categories. For any presentable $\infty$-category $\es{D}$, there is an equivalence
\[
\Fun(\underset{i \in I}{\colim}~\es{C}_i, \es{D}) \simeq \underset{i \in I}{\colim}~ \Fun(\es{C}_i, \es{D})
\]
between the $\infty$-category of $\es{D}$-valued functors on $\underset{i \in I}{\colim}~\es{C}_i$ (computed in $\sf{Cat}_\infty$) and the colimit (in $\sf{Pr}^L$) of the $\infty$-categories of $\es{D}$-valued functors on $\es{C}_i$.
\end{cor}

\section{Categories of \texorpdfstring{$G$-sets}{}}

Let $G$ be a topological group. For our purposes it suffices to consider the case where $G$ is a profinite group. We will denote by $\sf{Set}_G$ the category of sets with an action of the group $G$. Note that there is no finiteness assumptions on the underlying set. 

For $G$ any topological group and $N$ a normal subgroup of $G$, the canonical projection $G \to G/N$ Induces an adjunction 
\[
\adjunction{\inf_{G/N}^G}{\sf{Set}_{G/N}}{\sf{Set}_G}{(-)^N}.
\]
The left adjoint is frequently referred to as \emph{inflation} and is defined so that $N$ acts trivially, hence the unit map
\[
X \longrightarrow ({\inf}_{G/N}^G(X))^N,
\]
is equivalent to the identity. The counit
\[
{\inf}_{G/N}^G(X^N) \longrightarrow X,
\]
corresponds, as a map of sets, to the inclusion $X^N \hookrightarrow X$ and hence is an injection.

Let $G$ be a profinite group. The set $\{N \leq_o G\}$ of open and normal subgroups is a poset with respect to subgroup inclusion and the group $G$ is equivalently described as the direct (cofiltered) limit of the functor 
\[
(\{N \leq_o G\}, \subseteq) \longrightarrow \sf{Grp}, \quad N \longmapsto G/N.
\]
For example, the $p$-adic integers $\bb{Z}_p$ is the direct limit of the diagram 
\[\begin{tikzcd}
	\cdots & {\bb{Z}/p^3} & {\bb{Z}/p^2} & {\bb{Z}/p}
	\arrow[from=1-1, to=1-2]
	\arrow[from=1-2, to=1-3]
	\arrow[from=1-3, to=1-4]
\end{tikzcd}\]
induced by the canonical projections. For $M \subseteq N$, the canonical projection $G/M \to G/N$ induces an adjoint pair
\[
\adjunction{\inf_{G/N}^{G/M}}{\Set_{G/N}}{\Set_{G/M}}{(-)^M}.
\]
Since the projections are compatible, i.e., the diagram 
\[\begin{tikzcd}
	& {G/N} \\
	{G/M} && {G/K}
	\arrow[from=1-2, to=2-1]
	\arrow[from=1-2, to=2-3]
	\arrow[from=2-1, to=2-3]
\end{tikzcd}\]
commutes, the above adjunctions are compatible with the poset structure, i.e., the diagrams
\[\begin{tikzcd}
	& {\Set_{G/K}} &&&& {\Set_{G/K}} \\
	{\Set_{G/M}} && {\Set_{G/N}} && {\Set_{G/M}} && {\Set_{G/N}}
	\arrow["{(-)^M}"', from=1-2, to=2-1]
	\arrow["{(-)^N}", from=1-2, to=2-3]
	\arrow["{\sf{inf}_{G/M}^{G/K}}"', tail reversed, no head, from=1-6, to=2-5]
	\arrow["{\sf{inf}_{G/N}^{G/K}}", tail reversed, no head, from=1-6, to=2-7]
	\arrow["{(-)^N}"', from=2-1, to=2-3]
	\arrow["{\sf{inf}_{G/N}^{G/M}}"', tail reversed, no head, from=2-5, to=2-7]
\end{tikzcd}\]
commute for $K \subseteq M \subseteq N$ open and normal subgroups of $G$. As such, we may take a cofiltered limit along fixed points functors, and a filtered colimit along inflation functors. 

There are several subcategories of $\sf{Set}_G$ that will play a crucial role for us, and the above discussion descends to any of these subcategories. We will denote by $\sf{Fin}_G$ the category of finite $G$-sets. 

\begin{definition}
Let $G$ be a profinite group. A $G$-set $X$ is said to be \emph{discrete}\footnote{A more thorough discussion on discrete $G$-sets may be found in~\cite[Definition 3.1]{BarnesSugrueClassifying} and the ensuing discussion.} if the canonical map
\[
\underset{N \trianglelefteq_o G}{\colim}~X^N \longrightarrow X,
\]
is an isomorphism. Observe that a $G$-set $X$ is discrete if and only if the action of $G$ is continuous, when $X$ is equipped with the discrete topology. We denote the category of discrete $G$-sets by $\sf{Set}_G^\delta$ and the category of finite discrete $G$-sets by $\sf{Fin}_G^\delta$.
\end{definition}

Finite discrete $G$ sets may be described as a filtered colimit of the finite quotients.

\begin{thm}\label{thm: GSet as colim}
Let $G$ be a profinite group. The collection of inflation functors running over the set of open and normal subgroups of $G$ assembles into an equivalence of $\infty$-categories
\[
\underset{N \trianglelefteq_o G}{\colim}~\sf{Fin}_{G/N} \xrightarrow{\ \simeq \ } \sf{Fin}_G^\delta,
\]
between the colimit (in $\sf{Cat}_\infty$) of the $\infty$-categories of finite $G/N$-sets and the $\infty$-category of finite \emph{discrete} $G$-sets.
\end{thm}
\begin{proof}
Let $N$ be an open and normal subgroup of $G$. For $X$ a finite $G/N$-set, the inflation of $X$ to a $G$-set is still finite and discrete since each point of $\inf_{G/N}^G$ is fixed by $N$. As such, the collection of inflation functors assemble into a functor
\[
\inf: \underset{N \trianglelefteq_o G}{\colim}~\sf{Fin}_{G/N} \longrightarrow \sf{Fin}_G^\delta,
\]
which we will show is an equivalence of $\infty$-categories. To see that the induced functor is essentially surjective, observe that any finite discrete $G$-set $X$ is fixed by some open subgroup $H$, which must contain an open normal subgroup $N$. 

It is left to show that the functor is fully faithful. For a map $X \to Y$ of finite discrete $G$-sets, there exists an open and normal subgroup $N$ such that $X$ and $Y$ are both $N$-fixed. It follows that the map $X \to Y$ appears in term $N$ of the colimit. This construction provides an inverse to map induced by $\inf: \underset{N \trianglelefteq_o G}{\colim}~\sf{Fin}_{G/N} \longrightarrow \sf{Fin}_G^\delta$ on mapping sets. 
\end{proof}

We can extend the above result on finite (discrete) $G$-sets to all $G$-sets by Ind-completing, i.e., by formally adjoining filtered colimits. For an $\infty$-category $\es{C}$, we say that a functor $H: \es{C} \to \hat{\es{C}}$ exhibits $\hat{\es{C}}$ as the \emph{Ind-completion} of $\es{C}$ if $\hat{\es{C}}$ admits filtered colimits and for any $\infty$-category $\es{D}$ which admits filtered colimits, the induced functor
\[
\Fun^{\omega}(\hat{\es{C}}, \es{D}) \longrightarrow \Fun(\es{C}, \es{D}),
\]
is an equivalence between the $\infty$-category of finite colimit preserving functors from $\hat{\es{C}}$ to $\es{D}$ and the $\infty$-category of functors from $\es{C}$ to $\es{D}$, see e.g.,~\cite[\href{https://kerodon.net/tag/063K}{Definition 063K}]{kerodon}. The Ind-completion assembles into a functor
\[
\sf{Ind} : \sf{Cat}_\infty \longrightarrow \sf{Acc},
\]
from the $\infty$-category of $\infty$-categories to the $\infty$-category of accessible $\infty$-categories and accessible functors between them. If we assume that our input category has finite colimits, Ind-completion defines a functor
\[
\sf{Ind} : \sf{Cat}_\infty^{\rm rex} \longrightarrow \sf{Pr}^L,
\]
from the $\infty$-category of finitely cocomplete $\infty$-categories and right exact\footnote{A functor is \emph{right exact} if it preserves finite colimits.} functors to the $\infty$-category of presentable $\infty$-categories and left adjoints. By~\cite[Proposition 5.4.2.19 and Proposition 5.5.7.10]{HTT} this functor exhibits the latter category as a localization of the former, and hence Ind-completion (in this form) preserves colimits, see also~\cite[Proposition 5.3.5.10]{HTT}. Colimits in $\sf{Pr}^L$ are formed by passing to the diagram of right adjoints through the equivalence $(\sf{Pr}^{L})^\op \simeq \sf{Pr}^R$ and computing the limit. Since the forgetful functor $\sf{Pr}^R \to \widehat{\sf{Cat}}_\infty$ preserves limits~\cite[Theorem 5.5.3.18]{HTT}, we have the following description of the Ind-completion of a colimit of $\infty$-categories.

\begin{lem}\label{lem: Ind of colim}
Let $\es{C}_{(-)}: I \to \sf{Cat}_\infty^{\rm rex}$ be a diagram of finitely cocomplete $\infty$-categories. There is an equivalence of $\infty$-categories
\[
\sf{Ind}(\underset{i \in I}{\colim}~ \es{C}_i) \simeq \lim_{i \in I} \sf{Ind}(\es{C}_i),
\]
between the Ind-completion of the colimit (in $\sf{Cat}_\infty$) of $\es{C}_{(-)}$ and the limit (in $\sf{Pr}^R$) of the Ind-completion of $\es{C}_{(-)}$.
\end{lem}
\begin{proof}
The functor 
\[
\sf{Ind} : \sf{Cat}_\infty^{\rm rex} \longrightarrow \sf{Pr}^L,
\]
preserves colimits since it exhibits the latter category as a localization of the former. The result then follows by how colimits are computed in both categories: colimit in $ \sf{Cat}_\infty^{\rm rex}$ are computed in $\sf{Cat}_\infty$ while colimits in $\sf{Pr}^L$ are computed as limits in $\sf{Pr}^R$. 
\end{proof}

\begin{ex}\label{ex: Ind of fin}
Let $G$ be a finite group. The Ind-completion of the $\infty$-category $\Fin_G$ of finite $G$-sets is the $\infty$-category $\Set_G$ of all $G$-sets.
\end{ex}
\begin{proof}
By~\cite[\href{https://kerodon.net/tag/065F}{Remark 065F}]{kerodon}, Ind-completion commutes with taking the nerve, so it suffices to show that the $1$-category of $G$-sets is the Ind-completion (as $1$-categories) of the $1$-category $\Fin_G$. This completely classical statement follows from the observation that the inclusion $\sf{Fin}_G \hookrightarrow \sf{Set}_G$ is fully-faithful, finite $G$-sets are the compact objects in the category of $G$-sets and that every $G$-set may be written as a filtered colimit of finite $G$-sets.
\end{proof}

\begin{ex}\label{ex: Ind of profinite}
Let $G$ be a profinite group. The Ind-completion of the $\infty$-category $\Fin_G^\delta$ of finite discrete $G$-sets is the $\infty$-category $\Set_G^\delta$ of discrete $G$-sets.
\end{ex}
\begin{proof}
As above this reduces to the $1$-categorical statement which has an completely analogous proof using the colimit description of discrete $G$-sets.
\end{proof}

We obtain the following corollary of~\cref{thm: GSet as colim}.

\begin{cor}\label{cor: GSet as limit}
Let $G$ be a profinite group. The collection of fixed points functors assembles into an equivalence of $\infty$-categories 
\[
\Set_G^\delta \xrightarrow{\ \simeq \ } \underset{N \trianglelefteq_o G}{\lim}~\Set_{G/N}
\]
between the $\infty$-category of discrete $G$-sets and the limit (in $\sf{Pr}^R$) of the $\infty$-categories of $G/N$-sets.
\end{cor}
\begin{proof}
By~\cref{thm: GSet as colim} there is an equivalence
\[
\underset{N \trianglelefteq_o G}{\colim}~\Fin_{G/N} \simeq  \Fin_G^\delta
\]
between the colimit (in $\sf{Cat}_\infty$) of the $\infty$-categories of finite $G/N$-sets and the $\infty$-category of finite discrete $G$-sets. Applying Ind-completion to this equivalence provides an equivalence 
\[
\sf{Ind} \left( \underset{N \trianglelefteq_o G}{\colim}~\Fin_{G/N} \right) \simeq  \sf{Ind}(\Fin_G^\delta)
\]
between the Ind-completion of the colimit (in $\sf{Cat}_\infty$) of the $\infty$-categories of finite $G/N$-sets and the Ind-completion of the $\infty$-category of finite discrete $G$-sets. By~\cref{lem: Ind of colim} Ind-completion sends colimits in $\sf{Cat}_\infty^{\rm rex}$ to limits in $\sf{Pr}^R$, hence there is an equivalence
\[
\lim_{N \leq_{o} G} \sf{Ind}(\Fin_{G/N})  \simeq \sf{Ind}(\Fin_G^\delta)
\]
since the $\infty$-categories of finite $G/N$-sets are finitely cocomplete and the inflation functors are right exact. The result follows by using~\cref{ex: Ind of fin} to identify the left-hand side of the equivalence with the left-hand side of the claimed equivalence in the lemma statement, and analogously, using~\cref{ex: Ind of profinite} to identify the right-hand side. The identification of the functors involved follows from~\cref{lem: cocone in PrL is cone in PrR}.
\end{proof}

\section{span categories}

We make the following definition following~\cite[\S3.7]{Barwick}. A precise definition in terms of recording the $n$-simplicies of the simplicial set underlying the $\infty$-category may be found in \cite[Definition 3.6]{Barwick}. 

\begin{definition}\label{def: span category}
Let $\es{C}$ be an $\infty$-category that admits pullbacks. The \emph{span $\infty$-category $\sf{Span}(\es{C})$ of $\es{C}$} is the $\infty$-category with objects the objects of $\es{C}$ and mapping space given by
\[
\Map_{\sf{Span}(\es{C})}(X,Y) = \left( \es{C}_{/\{X,Y\}}\right)^\simeq.
\]
Composition 
\[
\left( \es{C}_{/\{X,Y\}}\right)^\simeq \times \left( \es{C}_{/\{Y,Z\}}\right)^\simeq \longrightarrow \left( \es{C}_{/\{X,Z\}}\right)^\simeq 
\]
is defined, up to coherent homotopy, by pullback. 
\end{definition}

\begin{rem}
For $G$ a finite group, it is typical to denote by $\sf{Span}(G)$ the span category of finite $G$-sets. When $G$ is a profinite group, it is standard (but less so than for finite groups) to denote by $\sf{Span}(G)$ the span category of finite discrete $G$-sets, see e.g.,~\cite[Example B]{Barwick}. Since we have introduced many variants of the category of $G$-sets for a profinite group $G$, we will not stick to this naming convention and rather always write $\sf{Span}(\es{C})$ for the span $\infty$-category of $\es{C}$.
\end{rem}

This construction assembles into a functor
\[
\sf{Span}(-): \sf{Cat}^{\rm lex} \longrightarrow \sf{Cat}_\infty,
\]
from the $\infty$-category of $\infty$-categories with finite limits and left exact\footnote{A functor is \emph{left exact} if it preserves finite limits.} functors between them, to the $\infty$-category of $\infty$-categories. We now give an alternative construction of this functor following work of Haugseng, Hebestreit, Linskens and Nuiten~\cite{HHLN}. A complementary point of view (with alternative motivation) is provided in the book-in-progress of Cnossen~\cite{Cnossen}.

\begin{definition}
An \emph{adequate triple} $(\es{C}, \es{C}_{\sf{B}}, \es{C}_{\sf{F}})$ consists of an $\infty$-category $\es{C}$ together with two wide\footnote{A \emph{wide} sub-$\infty$-category of an $\infty$-category $\es{C}$ is a subcategory containing all the objects of $\es{C}$. } sub-$\infty$-categories $\es{C}_{\sf{B}}$ and $\es{C}_{\sf{F}}$, whose morphisms are called \emph{backward} and \emph{forward}, respectively, such that
\begin{enumerate}
\item for any backward morphism $f: y \to x$ and any forward morphism $g: x' \to x$, there exists a pullback
\[\begin{tikzcd}
	{y'} & {x'} \\
	y & x
	\arrow["{f'}", from=1-1, to=1-2]
	\arrow["{g'}"', from=1-1, to=2-1]
	\arrow["g", from=1-2, to=2-2]
	\arrow["f"', from=2-1, to=2-2]
\end{tikzcd}\]
\item in any such pullback $f'$ is a backward morphism and $g'$ is a forward morphism.
\end{enumerate}
A functor $F: (\es{C}, \es{C}_{\sf{B}}, \es{C}_{\sf{F}}) \to (\es{D}, \es{D}_{\sf{B}}, \es{D}_{\sf{F}})$ of adequate triples is given by a functor $F: \es{C} \to \es{D}$ which preserves pullback squares in which the horizontal maps are backward and vertical maps forward. Denote by $\sf{AdTrip}$ the (non-full) sub-$\infty$-category of $\Fun(\Lambda_2[2], \sf{Cat}_\infty)$ spanned by the adequate triples and functors between them, see e.g.,~\cite[Definition 2.1]{HHLN}.
\end{definition}

The point of this definition is that adequate triples are the most general input for a span category. For our purposes we will only need the following fact: if $\es{C}$ admits pullbacks (e.g., is an object of $\sf{Cat}_\infty^{\rm lex}$), then $(\es{C}, \es{C}, \es{C})$ is an adequate triple. In particular, this gives a fully faithful embedding
\[
\sf{Cat}_\infty^{\rm lex} \hookrightarrow \sf{AdTrip}.
\]

\begin{thm}[{\cite[Theorem A]{HHLN}}]
There is an adjoint pair
\[
\adjunction{\sf{Tw}^r}{\sf{Cat}_\infty}{\sf{AdTrip}}{\sf{Span}(-)},
\]
in which $\sf{Tw}^r(\es{C})$ is the twisted arrow $\infty$-category\footnote{The twisted arrow $\infty$-category is the $\infty$-category of elements associated to the hom functor, see e.g.,~\cite[\href{https://kerodon.net/tag/046B}{Definition 046B}]{kerodon}. The superscript ``$r$'' indicates that we consider the version in which the projection maps are right fibrations, see e.g.,~\cite[\href{https://kerodon.net/tag/03JK}{Notation 03JK}]{kerodon}.} of $\es{C}$ viewed as an adequate triple via~\cite[Example 2.8]{HHLN}.
\end{thm}

It follows from~\cite[Theorem 2.13]{HHLN} that the restriction of the right adjoint
(for which we keep the same notation)
\[
\sf{Span}(-): \sf{Cat}_\infty^{\rm lex} \hookrightarrow \sf{AdTrip} \longrightarrow 
\sf{Cat}_\infty,
\]
agrees with Barwick's construction from~\cref{def: span category}. The $\infty$-category $\sf{AdTrip}$ has all limits which are computed as
\[
\lim_{i \in I}(\es{C}_i, (\es{C}_{i})_{\sf{B}}, (\es{C}_i)_{\sf{F}}) \simeq (\lim_{i \in I}\es{C}_i, \lim_{i \in I}(\es{C}_i)_{\sf{B}}, \lim_{i \in I}(\es{C}_i)_{\sf{F}}),
\]
where the limits on the right-hand side are computed in $\sf{Cat}_\infty$, see e.g.,~\cite[Lemma 2.4]{HHLN}. Said differently, the $\infty$-category of adequate triples has limits which are computed in $\Fun(\Lambda_2[2], \sf{Cat}_\infty)$, hence levelwise in $\sf{Cat}_\infty$.

\begin{lem}\label{lem: span preserves limits}
The functor 
\[
\sf{Span}(-) : \sf{Cat}_\infty^{\rm lex} \longrightarrow \sf{Cat}_\infty,
\]
preserves all limits.
\end{lem}
\begin{proof}
The functor in question is the composite 
\[
\sf{Cat}_\infty^{\rm lex} \hookrightarrow \sf{AdTrip} \xrightarrow{\sf{Span}(-)} \sf{Cat}_\infty.
\]
The left-hand functor preserves limits by the above discussion on how limits are computed in the $\infty$-category of adequate triples. The right-hand functor preserves limits as a right adjoint. It follows that the composite must also preserve limits.
\end{proof}

An immediate consequence of the above lemma is the following identification of the span category of discrete $G$-sets.

\begin{thm}\label{thm: span as limit}
Let $G$ be a profinite group. The collection of fixed point functors running over the set of open and normal subgroups of $G$ assembles into an equivalence of $\infty$-categories
\[
\sf{Span}(\Set_G^\delta) \xrightarrow{\ \simeq \ } \underset{N \trianglelefteq_o G}{\lim}~\sf{Span}(\Set_{G/N}),
\]
between the span $\infty$-category of discrete $G$-sets and the limit (in $\sf{Cat}_\infty$) of the span $\infty$-categories of $G/N$-sets.
\end{thm}
\begin{proof}
This follows by combining~\cref{cor: GSet as limit} and~\cref{lem: span preserves limits}, and by using~\cref{lem: cocone in PrL is cone in PrR} to identify the functors.
\end{proof}

\begin{rem}
The above limit only exists in the $\infty$-category $\sf{Cat}_\infty$. This is for two fundamental reasons: the $\infty$-category of spans is not presentable, and a left adjoint functor need not induce a left adjoint functor on the level of span categories. In fact, when an adjunction on underlying categories induces an adjunction on span categories is subtle, see~\cref{rem: adjoints}.
\end{rem}

The $\infty$-category of adequate triples also has filtered colimits and they are computed in the analogous way to limits, i.e., for a filtered $\infty$-category $I$ we have that
\[
\underset{i \in I}{\colim}~(\es{C}_i, (\es{C}_{i})_{\sf{B}}, (\es{C}_i)_{\sf{F}}) \simeq (\underset{i \in I}{\colim}~\es{C}_i, \underset{i \in I}{\colim}~(\es{C}_i)_{\sf{B}}, \underset{i \in I}{\colim}~(\es{C}_i)_{\sf{F}}).
\]
This follows essentially from the fact that on the level of mapping spaces, filtered colimits commute with pullbacks. The span functor also commutes with filtered colimits\footnote{There is another way to prove this result: observe that the statement in question is equivalent to the twisted arrow $\infty$-category $\sf{Tw}^r$ preserving compact objects and prove that an adequate triple  $(\es{C}, \es{C}_{\sf{B}}, \es{C}_{\sf{F}})$ is compact if $\es{C}$, $\es{C}_{\sf{B}}$ and, $\es{C}_{\sf{F}}$ are compact in $\sf{Cat}_\infty$. This observation follows from the morphism spaces in $\sf{AdTrip}$ being a finite limit.}.

\begin{lem}\label{lem: span preserves filtered colimits}
The functor 
\[
\sf{Span}(-) : \sf{Cat}_\infty^{\rm lex} \longrightarrow \sf{Cat}_\infty,
\]
preserves filtered colimits. 
\end{lem}
\begin{proof}
Following~\cite[Definition 2.12]{HHLN}, the functor in question may be deconstructed as a composite
\[
\sf{Cat}_\infty^{\rm lex} \hookrightarrow \sf{AdTrip} \xrightarrow{S_{\sf{Tw}^r}} \es{S}^{\Delta^\op} \xrightarrow{\sf{ac}} \sf{Cat}_\infty,
\]
where $\sf{ac}: \es{S}^{\Delta^\op} \to \sf{Cat}_\infty $ is the associated category functor which is left adjoint to the Rezk nerve $N: \sf{Cat}_\infty \to \es{S}^{\Delta^\op}$, and $S_{\sf{Tw}^r}$ is the singular complex functor associated to the twisted arrow $\infty$-category functor
\[
\sf{Tw}^r : \Delta \longrightarrow \sf{AdTrip}, \quad [n] \longmapsto \sf{Tw}^r(\Delta[n]).
\]
see e.g.,~\cite[Construction 2.10]{HHLN}.

It will suffice to show that each functor in this composite preserves filtered colimits. The right-hand functor is a left adjoint so preserves filtered colimits, and the left-hand functor preserves filtered colimits by the above discussion on how filtered colimits are computed for adequate triples, hence the proof reduces to the claim that the singular complex
\[
S_{\sf{Tw}^r} : \sf{AdTrip} \longrightarrow \es{S}^{\Delta^\op},
\]
preserves filtered colimits, i.e., that for a filtered diagram $\es{C}_{(-)}: I \to \sf{AdTrip}$, the canonical map 
\[
S_{\sf{Tw}^r}(\underset{i \in I}{\colim}~ \es{C}_i) \simeq \underset{i \in I}{\colim}~S_{\sf{Tw}^r}(\es{C}_i),
\]
is an equivalence of simplicial spaces. Since colimits of simplicial spaces are computed levelwise, it suffices to show that the canonical map
\[
S_{\sf{Tw}^r}(\underset{i \in I}{\colim}~ \es{C}_i)[n] \simeq \underset{i \in I}{\colim}~S_{\sf{Tw}^r}(\es{C}_i)[n],
\]
of spaces is an equivalence for each $[n] \in \Delta$.
Unravelling the definition we see that the singular complex functor is given by
\[
(\es{C}, \es{C}_{\sf{B}}, \es{C}_{\sf{F}}) \longmapsto ([n] \longmapsto \Map_{\sf{AdTrip}}(\sf{Tw}^r(\Delta[n]), (\es{C}, \es{C}_{\sf{B}}, \es{C}_{\sf{F}})) = Q_\bullet(\es{C})
\]
where we have chosen the compact notation $Q_\bullet(\es{C})$ in reference to the simplicial objects relationship with the $Q$-construction. The simplicial space $Q_\bullet(\es{C})$ is a complete Segal space by \cite[Theorem 2.13]{HHLN}, i.e., there are equivalences
\[
Q_n(\es{C}) \simeq Q_1(\es{C}) \times_{Q_0(\es{C})} \cdots \times_{Q_0(\es{C})} Q_1(\es{C})
\]
which are compatible with the face and degeneracy maps. Since filtered colimits commute with pullbacks in spaces, it suffices to show that the functors $Q_0(-)$ and $Q_1(-)$ commute with filtered colimits.

The functor $Q_0(-): \sf{AdTrip} \to \es{S}$ is equivalent to the composite
\[
\sf{AdTrip} \xrightarrow{ \rm forget} \sf{Cat}_\infty \xrightarrow{(-)^\simeq} \es{S},
\]
both of which commute with filtered colimits: the left-hand functor by the fact that filtered colimits of adequate triples are computed in $\sf{Cat}_\infty$ and the right-hand functor by the fact that the groupoid core functor may be described as $(-)^{\simeq} \simeq \Map_{\sf{Cat}_\infty}(\ast, -)$ since it is right adjoint to the inclusion $\es{S} \hookrightarrow \sf{Cat}_\infty$ and $\ast$ is compact as an $\infty$-category by~\cite[\href{https://kerodon.net/tag/0689}{Proposition 0689}]{kerodon}.

For the case of the functor $Q_1(-): \sf{AdTrip} \to \es{S}$ it suffices to show that given a map 
\[
\Lambda_2[2] \simeq \sf{Tw}^r([1]) \longrightarrow \underset{i \in I}{\colim}~ \es{C}_i
\]
of adequate triples, for a filtered diagram $\es{C}_{(-)}: I \to \sf{AdTrip}$, then the forward and backward legs land in $(\es{C}_i)_B$ and $(\es{C}_i)_F$ for some $i \in I$. This follows since filtered colimits of adequate triples are computed in $\sf{Cat}_\infty$.
\end{proof}

\begin{thm}\label{thm: span as colim}
Let $G$ be a profinite group. The collection of inflation functors running over the set of open and normal subgroups of $G$ assembles into an equivalence of $\infty$-categories
\[
\underset{N \trianglelefteq_o G}{\colim}~\sf{Span} (\Fin_{G/N}) \xrightarrow{\ \simeq \ } \sf{Span}(\Fin_G^\delta),
\]
between the colimit (in $\sf{Cat}_\infty$) of the span $\infty$-categories of finite $G/N$-sets and the span $\infty$-category of finite discrete $G$-sets.
\end{thm}
\begin{proof}
This follows by combining~\cref{thm: GSet as colim} and~\cref{lem: span preserves filtered colimits} and using \cref{lem: cocone in PrL is cone in PrR} to identify the functors.
\end{proof}

\section{Mackey functors for profinite groups}

As an application of the theory we consider Mackey functors for profinite groups. For this we need to introduce some terminology. An $\infty$-category $\es{C}$ is \emph{semiadditive} if it has a notion of direct sums, i.e., if $\es{C}$ is pointed, has both products and coproducts, and the canonical map
\[
\coprod_{i \in I} C_i \longrightarrow \prod_{i \in I} C_i,
\]
from a finite coproduct of objects of $\es{C}$ to a finite product of objects of $\es{C}$ is an equivalence. We denote by $\sf{Cat}_\infty^{\rm sadd}$ the sub-$\infty$-category of $\sf{Cat}_\infty$ spanned by semiadditive $\infty$-categories and biproduct preserving functors. An $\infty$-category $\es{C}$ is said to be \emph{disjunctive} if it has finite limits, finite coproducts and the latter are disjoint and universal, i.e., if the functor
\[
\prod_{i \in I} \es{C}_{/ C_i} \longrightarrow  \es{C}_{/ \coprod_{i \in I}C_i}
\]
given by the coproduct to be an equivalence, see e.g.,~\cite[Definition 4.2]{Barwick} and the references therein. We denote by $\sf{Cat}_\infty^{\rm disj}$ the sub-$\infty$-category of $\sf{Cat}_\infty$ spanned by the disjunctive $\infty$-categories and functors which preserve pullbacks and finite coproducts.

The key reason for considering disjunctive $\infty$-categories is that the span functor sends disjunctive $\infty$-categories to semiadditive $\infty$-categories, see e.g.,~\cite[Proposition 4.3]{Barwick}. This allows us to make the following definition, following~\cite[Definition 6.1]{Barwick}, see also~\cite[Notation 6.3]{Barwick}, keeping in mind that loc. cit. deals with the situation in much more generality that we are currently.

\begin{definition}
Let $\es{C}$ be a disjunctive $\infty$-category and $\es{E}$ a semiadditive $\infty$-category. The \emph{$\infty$-category $\sf{Mack}(\es{C}; \es{E}) $ of $\es{E}$-valued Mackey functors on $\es{C}$} to be the $\infty$-category $\Fun^\oplus (\sf{Span}(\es{C}), \es{E})$ of (co)product preserving functors\footnote{The $\infty$-category $\Fun^\oplus (\sf{Span}(\es{C}), \es{E})$ is defined as the full subcategory of $\Fun(\sf{Span}(\es{C}), \es{E})$ spanned by the (co)product preserving functors.} from the span category of $\es{C}$ to $\es{E}$
\end{definition}

The $\infty$-category of $\es{E}$-valued Mackey functors on $\es{C}$ has all limits and colimits that $\es{E}$ has, and they are computed levelwise, see e.g.,~\cite[Corollary 6.5.1]{Barwick}. Assuming $\es{E}$ is bicomplete, the $\infty$-category of Mackey functors is a smashing\footnote{Technically, \cite[Lemma 3.5]{BarwickGlasmanShah} only proves that the $\infty$-category of Mackey functors is a localization. To see that it is smashing, i.e., that the localization functor preserves all colimits, one has to observe that the fully faithful inclusion preserves all colimits, but this last is precisely the content of~\cite[Proposition 6.5]{Barwick}.} localization~\cite[Lemma 3.5]{BarwickGlasmanShah} of the $\infty$-category $\Fun(\sf{Span}(\es{C}), \es{E})$. As such, the $\infty$-category of $\es{E}$-valued Mackey functors in $\es{C}$ is presentable if $\es{E}$ is presentable. From now on, we will assume that $\es{E}$ is presentable and semiadditive. A functor $F: \es{C} \to \es{D}$ between disjunctive $\infty$-categories (which necessarily preserves finite coproducts and pullbacks) induces an adjoint pair
\[
\adjunction{F_!}{\sf{Mack}(\es{C}; \es{E})}{\sf{Mack}(\es{C}; \es{E})}{F^\ast},
\]
in which the right adjoint is given by restriction along the induced functor 
\[
\sf{Span}(F): \sf{Span}(\es{C}) \to \sf{Span}(\es{D}),
\]
and the left adjoint exists via the Adjoint Functor Theorem~\cite[Corollary 5.5.2.9]{HTT}, see e.g.,~\cite[Notation 6.9]{Barwick}. 

\begin{rem}\label{rem: adjoints}
Suppose $F: \es{C} \to \es{D}$ as above has a right adjoint $G: \es{D} \to \es{C}$. In general, $\sf{Span}(F)$ need not be left adjoint to $\sf{Span}(G)$ and hence $F^\ast$ need not be equivalent to $G_!$ on the level of Mackey functors. In fact, $\sf{Span}(F)$ is left adjoint to $\sf{Span}(G)$ if the squares 
\[\begin{tikzcd}
	X & {GF(X)} && {FG(Y)} & Y \\
	{X'} & {GF(X')} && {FG(Y')} & {Y'}
	\arrow["\eta", from=1-1, to=1-2]
	\arrow["f"', from=1-1, to=2-1]
	\arrow["{GF(f)}", from=1-2, to=2-2]
	\arrow["\varepsilon", from=1-4, to=1-5]
	\arrow["{FG(g)}"', from=1-4, to=2-4]
	\arrow["g", from=1-5, to=2-5]
	\arrow["\eta", from=2-1, to=2-2]
	\arrow["\varepsilon", from=2-4, to=2-5]
\end{tikzcd}\]
are pullbacks for every map $f: X \to X'$ in $\es{C}$ and $g: Y \to Y'$ in $\es{D}$. We find ourselves in the situation where we do not have adjoint pairs at the level of span categories when considering the adjunction between inflation and fixed points functors: for the adjunction
\[
\adjunction{\inf_{G/N}^G}{\Fin_{G/N}}{\Fin_G}{(-)^N}, 
\]
the corresponding diagrams are
\[\begin{tikzcd}
	X & {(\inf_{G/N}^GX)^N} && {\inf_{G/N}^G(Y^N)} & Y \\
	{X'} & {(\inf_{G/N}^GX')^N} && {\inf_{G/N}^G((Y')^N)} & {Y'}
	\arrow["\eta", from=1-1, to=1-2]
	\arrow["f"', from=1-1, to=2-1]
	\arrow["{(\inf_{G/N}^G(f))^N}", from=1-2, to=2-2]
	\arrow["\varepsilon", from=1-4, to=1-5]
	\arrow["{\inf_{G/N}^G(g^N)}"', from=1-4, to=2-4]
	\arrow["g", from=1-5, to=2-5]
	\arrow["\eta", from=2-1, to=2-2]
	\arrow["\varepsilon", from=2-4, to=2-5]
\end{tikzcd}\]
where left-hand square above is a pullback since the unit is an isomorphism, but the right-hand square need not be a pullback since the counit is only an injection.
\end{rem}

The aim of this section is to provide a (co)limit description for the category of Mackey functors on a profinite group. Since Mackey functors are defined in terms of product preserving functors, we will require the following lemma. 

\begin{lem}\label{lem: colim prod preserving}
Let $\es{E}$ be a (possibly large) $\infty$-category with finite products and let $\es{C}_{(-)}: I \to \sf{Cat}_\infty^{\times}$ be a filtered diagram of small $\infty$-categories with finite products. A functor $F: \underset{i \in I}{\colim}~\es{C}_i \to \es{E}$ is product preserving if and only if each component $F_i : \es{C}_i \to \es{E}$ is product preserving. 
\end{lem}
\begin{proof}
Let $\es{C}_{(-)} : I \to \sf{Cat}_\infty^\times$ be a filtered diagram of $\infty$-categories with (finite) products and product preserving maps. The colimit $\underset{i \in I}{\colim}~\es{C}_i$ inherits products as follows. For objects $X$ and $Y$ in $\underset{i \in I}{\colim}~\es{C}_i$, there exists $i,j \in I$ such that $X = \eta_i(X_i)$ and $Y = \eta_{j}(Y_j)$ for some $X_i \in \es{C}_i$ and $Y_{j} \in \es{C}_j$, where $\eta_k : \es{C}_k \to \underset{i \in I}{\colim}~\es{C}_i$ is the canonical map. Since $I$ is filtered, there exists (at least one) $k \in I$ with $i,j \leq k$ such that $X= \eta_k(X_k)$ and $Y = \eta_k(Y_k)$ for $X_k, Y_k \in \es{C}_k$. The product $X \times Y$ is then given by
\[
X \times Y = \eta_k(X_k \times Y_k).
\]
In particular, by definition, the canonical maps 
\[
\eta_k : \es{C}_k \longrightarrow \underset{i \in I}{\colim}~\es{C}_i
\]
are product preserving. 

Now consider the composite:
\[\begin{tikzcd}
	{\es{C}_i } & {\underset{i \in I}{\colim}~\es{C}_i} & {\es{E}}
	\arrow["{\eta_i}", from=1-1, to=1-2]
	\arrow["{F_i}"', shift right, curve={height=18pt}, from=1-1, to=1-3]
	\arrow["F", from=1-2, to=1-3]
\end{tikzcd}\]
which realises the equivalence of~\cref{lem: functor cats and colims}. If $F$ is product preserving, then so is $F_i$ for each $i \in I$ as product preserving maps are closed under composition. 

For the converse, assume for each $i \in I$ that the map $F_i : \es{C}_i \to \es{E}$ is product preserving. Then for $X$ and $Y$ objects of $\underset{i \in I}{\colim}~\es{C}_i$, there exits $k \in I$ such that
\begin{align*}
F(X \times Y) &\simeq F(\eta_k(X_k \times Y_k)) \\
&\simeq  F_k(X_k \times Y_k) \\
&\simeq  F_k(X_k) \times F_k(Y_k) \\
&\simeq  F(\eta_k(X_k)) \times F(\eta_k(X_k)) \\ &\simeq  F(X) \times F(Y)  
\end{align*}
where the first equivalence follows by how products are computed in the colimit category, the second equivalence is the definition of $F_k$ as a composite, the third equivalence uses the fact that $F_k$ is assumed to be product preserving, the penultimate equivalence follows again from the definition of $F_k$ as a composite and the finial equivalence follows from the describe of the objects of the colimit category, when the indexed category is filtered, see~\cref{subsection: colimits}.
\end{proof}

The technical heart of the proof of the main theorem of this section (\cref{thm: Mackey functors as a limit}) is the following observation that the functor $\Fun^\oplus(-, \es{E})$ preserves limits.

\begin{lem}\label{lem: AdFun preserves lims}
Let $\es{E}$ be a presentable semiadditive $\infty$-category and let $\es{C}_{(-)}: I \to \sf{Cat}_\infty^{\rm sadd}$ be a filtered diagram of semiadditive $\infty$-categories. There is an equivalence of $\infty$-categories
\[
\Fun^\oplus(\underset{i \in I}{\colim}~ \es{C}_i,  \es{E}) \simeq \lim_{i \in I} \Fun^\oplus(\es{C}_i, \es{E})
\]
between $\es{E}$-valued product-preserving functors on the colimit (in $\sf{Cat}_\infty$) of $\es{C}_{(-)}$ and the limit (in $\sf{Pr}^R$) of the categories of $\es{E}$-valued functors on $\es{C}_i$. 
\end{lem}
\begin{proof}
Consider the commutative diagram
\[\begin{tikzcd}
	{\Fun(\underset{i \in I}{\colim}~\es{C}_i, \es{E}) } & {\underset{i \in I}{\lim}~\Fun(\es{C}_i, \es{E}) } \\
	{\Fun^\oplus(\underset{i \in I}{\colim}~\es{C}_i, \es{E}) } & {\underset{i \in I}{\lim}~\Fun^\oplus(\es{C}_i, \es{E}) }
	\arrow["\simeq", from=1-1, to=1-2]
	\arrow[hook, from=2-1, to=1-1]
	\arrow[from=2-1, to=2-2]
	\arrow[hook, from=2-2, to=1-2]
\end{tikzcd}\]
in which the vertical arrows are fully faithful inclusions and the top-most horizontal arrow is an equivalence by~\cref{lem: functor cats and colims}. We wish to show that the lower horizontal arrow is an equivalence of (possibly large) $\infty$-categories.

For essential surjectivity, let $\{F_i\}_{i \in I}$ be an object in $\lim_{i \in I}\Fun^\oplus(\es{C}_i, \es{E})$ i.e., for each $i \in I$, the functor $F_i: \es{C}_i \to \es{E}$ preserves products and the collection of such functors forms a cocartesian section in the sense of~\cref{subsection: limits}. By embedding the cocartesian section $\{F_i\}_{i \in I}$ into the larger category $\lim_{i \in I}\Fun(\es{C}_i, \es{E})$, we find a functor
\[
F: \underset{i \in I}{\colim}~\es{C}_i \longrightarrow \es{E},
\]
whose image under the top equivalence is precisely $\{F_i\}_{i \in I}$. The functor $F$ is product preserving by~\cref{lem: colim prod preserving}, hence lifts to an object in $\Fun^\oplus(\underset{i \in I}{\colim}~\es{C}_i, \es{E})$ so that the bottom map is essentially surjective.

To see that the bottom arrow is fully faithful, observe that the composite 
\[
\Fun^\oplus(\underset{i \in I}{\colim}~\es{C}_i, \es{E}) \hookrightarrow \Fun(\underset{i \in I}{\colim}~\es{C}_i, \es{E}) \xrightarrow{ \ \simeq \ } \lim_{i \in I}\Fun(\es{C}_i, \es{E}) 
\]
is fully faithful as the composite of two fully faithful functors. As such, it suffices to prove that the image of this composite lands in $\lim_{i \in I}\Fun^\oplus(\es{C}_i, \es{E})$. This follows from~\cref{lem: colim prod preserving} since $F: \underset{i \in I}{\colim}~\es{C}_i \longrightarrow \es{E}$ is product preserving if and only if each composite
\[\begin{tikzcd}
	{\es{C}_i } & {\underset{i \in I}{\colim}~\es{C}_i} & {\es{E}}
	\arrow["{\eta_i}", from=1-1, to=1-2]
	\arrow["{F_i}"', shift right, curve={height=18pt}, from=1-1, to=1-3]
	\arrow["F", from=1-2, to=1-3]
\end{tikzcd}\]
is product preserving. 
\end{proof}

\begin{lem}\label{lem: semiadditive}
Let $\es{E}$ be a semiadditive $\infty$-category and let $\es{C}_{(-)}: I \to \sf{Cat}_\infty^{\rm disj}$ be a filtered diagram of disjunctive $\infty$-categories. There is an equivalence of $\infty$-categories
\[
\sf{Mack}(\underset{i \in I}{\colim}~ \es{C}_i~; \es{E}) \simeq \lim_{i \in I} \sf{Mack}(\es{C}_i~; \es{E})
\]
between $\es{E}$-valued Mackey functors on the colimit (in $\sf{Cat}_\infty$) of $\es{C}_{(-)}$ and the limit (in $\sf{Pr}^R$) of the categories of $\es{E}$-valued Mackey functors on $\es{C}_i$. 
\end{lem}
\begin{proof}
Recall that the $\infty$-category of $\es{E}$-valued Mackey functor on a disjunctive $\infty$-category $\es{C}$ is defined as
\[
\sf{Mack}(\es{C}~; \es{E}) = \Fun^\oplus(\sf{Span}(\es{C}), \es{E}).
\]
Given a filtered diagram 
\[
\es{C}_{(-)}: I \to \sf{Cat}_\infty^{\rm disj}, \quad i \longmapsto \es{C}_i
\]
of disjunctive $\infty$-categories,~\cref{lem: span preserves filtered colimits} provides a filtered diagram 
\[
\sf{Span}(\es{C}_{(-)}): I \to \sf{Cat}_\infty^{\rm sadd}, \quad i \longmapsto \sf{Span}(\es{C}_i)
\]
of semiadditive $\infty$-categories since the inclusion $\sf{Cat}_\infty^{\rm sadd} \subseteq \sf{Cat}_\infty$ preserves filtered colimits. To see this last observation, note that the inclusion factors
\[
\sf{Cat}_\infty^{\rm sadd} \hookrightarrow \sf{Cat}_\infty^{\times} \hookrightarrow \sf{Cat}_\infty
\]
through the $\infty$-category $\sf{Cat}_\infty^{\times}$ of $\infty$-categories with finite products. The inclusion
\[
\sf{Cat}_\infty^{\rm sadd} \hookrightarrow \sf{Cat}_\infty^{\times}
\]
preserves all colimits since it is a left adjoint, with right adjoint given by taking commutative monoid objects. The inclusion 
\[
\sf{Cat}_\infty^{\times} \hookrightarrow \sf{Cat}_\infty
\]
preserves filtered colimits by the formula for the product in a filtered colimit of $\infty$-categories with finite products, and hence the composite inclusion also preserves filtered colimits. The proof is now complete by an application of~\cref{lem: AdFun preserves lims}.
\end{proof}

We now specialise to the case of interest, namely profinite groups. 

\begin{thm}\label{thm: Mackey functors as a limit}
Let $\es{E}$ be a semiadditive $\infty$-category and let $G$ be a profinite group. The collection of categorical fixed point functors running over the set of open and normal subgroups of $G$ assembles into an equivalence of $\infty$-categories
\[
\sf{Mack}(\Fin^\delta_G~; \es{E}) \xrightarrow{\ \simeq \ } \lim_{N \leq_o G} \sf{Mack}(\Fin_{G/N}~; \es{E})
\]
between the $\infty$-category of $G$-Mackey functors and the limit (in $\sf{Pr}^R$) of the $\infty$-categories of $G/N$-Mackey functors.
\end{thm}
\begin{proof}
Combine~\cref{lem: semiadditive} with~\cref{thm: GSet as colim}. To identify the functors as categorical fixed points combine~\cref{lem: cocone in PrL is cone in PrR} with~\cite[\S B.7]{Barwick}.
\end{proof}

Taking $\es{E} =\s$ to be the $\infty$-category of spectra yields the following corollary upon identifying $G$-spectra with the category of spectral $G$-Mackey functors~\cite{Barwick, GuillouMay}. This gives an alternative proof of the existence of a \emph{continuous} model for the category of $G$-spectra due to the first author, Balchin and Barthel.~\cite[Theorem 6.5]{BBB}.

\begin{cor}
Let $G$ be a profinite group. The collection of categorical fixed point functors running over the set of open and normal subgroups of $G$ assembles into an equivalence of $\infty$-categories
\[
\s_G \xrightarrow{ \ \simeq \ } \lim_{N \leq_o G} \s_{G/N}
\]
between the $\infty$-category of $G$-spectra and the limit (in $\sf{Pr}^R$) of the $\infty$-categories of $G/N$-spectra.
\end{cor}

\printbibliography

@book{BachmannHoyois,
 author = {Bachmann, T. and Hoyois, M.},
 title = {Norms in motivic homotopy theory},
 fseries = {Ast{\'e}risque},
 series = {Ast{\'e}risque},
 issn = {0303-1179},
 volume = {425},
 isbn = {978-2-85629-939-5},
 year = {2021},
 publisher = {Paris: Soci{\'e}t{\'e} Math{\'e}matique de France (SMF)},
 language = {English},
 doi = {10.24033/ast.1147},
 zbMATH = {7403459},
 Zbl = {1522.14028}
}

@misc{BBB,
  title={Profinite equivariant spectra and their tensor-triangular geometry},
  author={Balchin, S. and Barnes, D. and Barthel, T.},
      year={2024},
      eprint={2401.01878},
      archivePrefix={arXiv},
      primaryClass={math.AT},
}

@article{BarnesSugrueClassifying,
 author = {Barnes, D. and Sugrue, D.},
 title = {Classifying rational {{\(G\)}}-spectra for profinite {{\(G\)}}},
 fjournal = {Algebraic \& Geometric Topology},
 journal = {Algebr. Geom. Topol.},
 issn = {1472-2747},
 volume = {24},
 number = {7},
 pages = {3801--3825},
 year = {2024},
 language = {English},
 doi = {10.2140/agt.2024.24.3801},
 zbMATH = {7962282}
}

@article{Barwick,
 author = {Barwick, C.},
 title = {Spectral {Mackey} functors and equivariant algebraic {{\(K\)}}-theory. {I}.},
 fjournal = {Advances in Mathematics},
 journal = {Adv. Math.},
 issn = {0001-8708},
 volume = {304},
 pages = {646--727},
 year = {2017},
 language = {English},
 doi = {10.1016/j.aim.2016.08.043},
 zbMATH = {6642269},
 Zbl = {1348.18020}
}

@article{BarwickGlasmanShah,
 author = {Barwick, C. and Glasman, S. and Shah, J.},
 title = {Spectral {Mackey} functors and equivariant algebraic {{\(K\)}}-theory. {II}.},
 fjournal = {Tunisian Journal of Mathematics},
 journal = {Tunis. J. Math.},
 issn = {2576-7658},
 volume = {2},
 number = {1},
 pages = {97--146},
 year = {2020},
 language = {English},
 doi = {10.2140/tunis.2020.2.97},
 zbMATH = {7074072},
 Zbl = {1461.18009}
}

@article{BGT,
 author = {Blumberg, A. J. and Gepner, D. and Tabuada, G.},
 title = {A universal characterization of higher algebraic {{\(K\)}}-theory},
 fjournal = {Geometry \& Topology},
 journal = {Geom. Topol.},
 issn = {1465-3060},
 volume = {17},
 number = {2},
 pages = {733--838},
 year = {2013},
 language = {English},
 doi = {10.2140/gt.2013.17.733},
 zbMATH = {6172909},
 Zbl = {1267.19001}
}

@article{HermitianI,
 author = {Calm{\`e}s, B. and Dotto, E. and Harpaz, Y. and Hebestreit, F. and Land, M. and Moi, K. and Nardin, D. and Nikolaus, T. and Steimle, W.},
 title = {Hermitian {K}-theory for stable {{\(\infty\)}}-categories. {I}: {Foundations}},
 fjournal = {Selecta Mathematica. New Series},
 journal = {Sel. Math., New Ser.},
 issn = {1022-1824},
 volume = {29},
 number = {1},
 pages = {269},
 note = {Id/No 10},
 year = {2023},
 language = {English},
 doi = {10.1007/s00029-022-00758-2},
 zbMATH = {7625214},
 Zbl = {1522.19001}
}

@misc{HermitianII,
 author = {Calm{\`e}s, B. and Dotto, E. and Harpaz, Y. and Hebestreit, F. and Land, M. and Moi, K. and Nardin, D. and Nikolaus, T. and Steimle, W.},
     TITLE = {Hermitian {K}-theory for stable {$\infty$}-categories
              {II}: {C}obordism categories and additivity},
   JOURNAL = {Acta Math.},
  FJOURNAL = {Acta Mathematica},
    VOLUME = {235},
      YEAR = {2026},
    NUMBER = {2},
     PAGES = {149--400},
      ISSN = {0001-5962,1871-2509},
   MRCLASS = {19G38 (18N99 55)},
  MRNUMBER = {5009505},
       DOI = {10.4310/acta.2025.n235.n2.a1},
}

@misc{HermitianIII,
 author = {Calm{\`e}s, B. and Dotto, E. and Harpaz, Y. and Hebestreit, F. and Land, M. and Moi, K. and Nardin, D. and Nikolaus, T. and Steimle, W.},
 title = {Hermitian {K}-theory for stable ${{\infty}}$-categories {III}: {Grothendieck}-{Witt} groups of rings},
      year={2021},
      eprint={2009.07225},
      archivePrefix={arXiv},
      primaryClass={math.KT},
    howpublished = {(To appear in Annals of Mathematics)}
}

@misc{Cnossen,
  author       = {Cnossen, B.},
  title        = {Stable Homotopy Theory and Higher Algebra},
  howpublished = {Available from the authors website: \url{https://sites.google.com/view/bastiaan-cnossen/home}},
  year         = {2025},
}

@article{Efimov,
  title={K-theory and localizing invariants of large categories},
  author={Efimov, A. I},
      year={2025},
      eprint={2405.12169},
      archivePrefix={arXiv},
      primaryClass={math.KT},
}

@article{GuillouMay,
 author = {Guillou, Bertrand J. and May, J. Peter},
 title = {Models of {{\(G\)}}-spectra as presheaves of spectra},
 fjournal = {Algebraic \& Geometric Topology},
 journal = {Algebr. Geom. Topol.},
 issn = {1472-2747},
 volume = {24},
 number = {3},
 pages = {1225--1275},
 year = {2024},
 language = {English},
 doi = {10.2140/agt.2024.24.1225},
 zbMATH = {7901539},
 Zbl = {1551.55007}
}

@article{Harpaz,
 author = {Harpaz, Y.},
 title = {Ambidexterity and the universality of finite spans},
 fjournal = {Proceedings of the London Mathematical Society. Third Series},
 journal = {Proc. Lond. Math. Soc. (3)},
 issn = {0024-6115},
 volume = {121},
 number = {5},
 pages = {1121--1170},
 year = {2020},
 language = {English},
 doi = {10.1112/plms.12367},
 zbMATH = {7242523},
 Zbl = {1479.55031}
}

@article{HHLN,
 author = {Haugseng, R. and Hebestreit, F. and Linskens, S. and Nuiten, J.},
 title = {Two-variable fibrations, factorisation systems and {{\(\infty\)}}-categories of spans},
 fjournal = {Forum of Mathematics, Sigma},
 journal = {Forum Math. Sigma},
 issn = {2050-5094},
 volume = {11},
 pages = {70},
 note = {Id/No e111},
 year = {2023},
 language = {English},
 doi = {10.1017/fms.2023.107},
 zbMATH = {7781651},
 Zbl = {1529.18019}
}

@article{Lindner,
 author = {Lindner, H.},
 title = {A remark on {Mackey}-functors},
 fjournal = {Manuscripta Mathematica},
 journal = {Manuscr. Math.},
 issn = {0025-2611},
 volume = {18},
 pages = {273--278},
 year = {1976},
 language = {English},
 doi = {10.1007/BF01245921},
 zbMATH = {3501745},
 Zbl = {0321.18002}
}

@article{LNP,
 author = {Linskens, S. and Nardin, D. and Pol, L.},
 title = {Global homotopy theory via partially lax limits},
 fjournal = {Geometry \& Topology},
 journal = {Geom. Topol.},
 issn = {1465-3060},
 volume = {29},
 number = {3},
 pages = {1345--1440},
 year = {2025},
 language = {English},
 doi = {10.2140/gt.2025.29.1345},
 zbMATH = {8055644}
}

@misc{kerodon,
  author       = {Lurie, J.},
  title        = {Kerodon},
  howpublished = {\url{https://kerodon.net}},
  year         = {2025},
}

@book{HTT,
 author = {Lurie, J.},
 title = {Higher topos theory},
 fseries = {Annals of Mathematics Studies},
 series = {Ann. Math. Stud.},
 volume = {170},
 isbn = {978-0-691-14049-0; 978-0-691-14048-3},
 year = {2009},
 publisher = {Princeton, NJ: Princeton University Press},
 language = {English},
 doi = {10.1515/9781400830558},
 zbMATH = {5497319},
 Zbl = {1175.18001}
}

@misc{Rozenblyum,
  author       = {Rozeblyum, N.},
  title        = {Filtered colimits of $\infty$-categories},
  howpublished = {\url{https://www.math.toronto.edu/nick/notes/colimits.pdf}},
  year         = {2012},
}

@article{Scholze,
  title={Six-functor formalisms},
  author={Scholze, P.},
      year={2025},
      eprint={2510.26269},
      archivePrefix={arXiv},
      primaryClass={math.AG},
}

@article {Fausk,
    AUTHOR = {Fausk, H.},
     TITLE = {Equivariant homotopy theory for pro-spectra},
   JOURNAL = {Geom. Topol.},
  FJOURNAL = {Geometry \& Topology},
    VOLUME = {12},
      YEAR = {2008},
    NUMBER = {1},
     PAGES = {103--176},
      ISSN = {1465-3060,1364-0380},
   MRCLASS = {55P91 (18G55)},
  MRNUMBER = {2377247},
MRREVIEWER = {J.\ P. C. Greenlees},
       DOI = {10.2140/gt.2008.12.103},
}
\end{document}